\documentclass{article}
\usepackage{amssymb,amsmath}
\usepackage{amsthm}
\usepackage{hyperref}
\usepackage{amscd}
\title{Mod $p$ Monodromy of Cyclic Covers of the Projective Line}
\author{Stepan Nesterov}
\newtheorem{theorem}{Theorem}
\newtheorem{proposition}[theorem]{Proposition}
\newtheorem{lemma}{Lemma}

\theoremstyle{definition}

\newtheorem{definition}{Definition}[section]

\newtheorem*{notation}{Notation}
\newtheorem*{remark}{Remark}
\DeclareMathOperator{\Gal}{\mathrm{Gal}}

\DeclareMathOperator{\GL}{\mathrm{GL}}
\DeclareMathOperator{\PGL}{\mathrm{PGL}}

\DeclareMathOperator{\PSL}{\mathrm{PSL}}
\DeclareMathOperator{\U}{\mathrm{U}}
\DeclareMathOperator{\SU}{\mathrm{SU}}
\DeclareMathOperator{\SlU}{\mathrm{S}_{\mathit{l}}\mathrm{U}}
\DeclareMathOperator{\SlL}{\mathrm{S}_{\mathit{l}}\mathrm{L}}
\DeclareMathOperator{\PSU}{\mathrm{PSU}}
\DeclareMathOperator{\SL}{\mathrm{SL}}
\DeclareMathOperator{\Sp}{\mathrm{Sp}}

\DeclareMathOperator{\et}{\mathrm{\acute{e}t}}
\begin{document}
	\maketitle
	\begin{abstract}
		In this paper, we prove a big monodromy theorem for the monodromy of cyclic coverings of projective line for cohomology with $\mathbf{F}_p$-coefficients. This is a direct generalization of the results of Achter and Pries, where such a theorem is proved for cyclic coverings of degree $2$ and $3$. Instead of generalizing their methods, we adapt the proof of the analogous theorem for integral cohomology. In our subsequent work, we will apply this theorem to construct in infinitely many cases Galois extensions of $\mathbf{Q}$ with Galois group $\PSL(n,q)$ and $\PSU(n,q)$, where $q$ can be an arbitrarilty large prime power.
	\end{abstract}
	\section{Introduction}
	Over $\mathbb{C}$, cyclic branched coverings of $\mathbf{P}^1$ can be explicitly constructed as smooth proper models of affine curves given by equations of the form $y^l = \prod_{i=0}^n (x-a_i)^{k_i}$ for some distinct scalars $a_i \in \mathbb{C}$, and integers $k_i \in \{1, 2, \ldots, l-1 \}$. Let us denote such a covering by $X_{\bar{a}, \bar{k}} \to  \mathbf{P}^1$. The tuple $(k_0, k_1, \ldots, k_n)$ will be called the \emph{monodromy vector} and fixed throughout the argument. The set of all tuples $(a_0, a_1, \ldots, a_n)$ of distinct points of $\mathbb{C}$ will be called \emph{the configuration space} and denoted by $\mathcal{C}_n$. It is clear that all the cyclic branched coverings of $\mathbf{P}^1$ with the monodromy vector $\bar{k}$ can be considered as fibers of a universal family $\mathcal{X}_{\bar{k}} \to \mathcal{C}_n$. It is well-known that the fundamental group of $\mathcal{C}_n$ is canonically isomorphic to the Artin's pure braid group $P_{n+1}$ of pure braids on $n+1$ strands, and so we obtain a representation of $P_{n+1}$ on the cohomology group of any fiber of the universal cyclic branched covering. \par 
	The case of the cohomology group with integral coeffients was studied in \cite{cyclic}, where the following result was obtained:
	\begin{theorem} \cite[Theorem 2]{cyclic}
		If $l \geq 3$, $n \geq 2l$, then the image of the monodromy representation of $P_{n+1}$ on $H^1(X_{\bar{a}}; \mathbf{Z})$ is a finite index subgroup in $\U_{n-1}(\mathbf{Z}[\cos \frac{2 \pi}{l}])$.
	\end{theorem}
	We will modify this result to work for the cohomology group with mod $p$ coefficients instead. For the convenience of reader, let us recall what the possible behaviour of the reduction mof $p$ of the unitary group is. \par 
	A unitary group over the ring $R=\mathbf{Z}[\cos \frac{2 \pi}{l}]$ is defined as follows: We take a quadratic algebra $\mathbf{Z}[\zeta_l]$ over $R$, with complex conjugation as an involution, a finite free module over $\mathbf{Z}[\zeta_l]$ and a nondegenerate hermitian form with respect to this involution. The unitary group is the isometry group of this form. Now let us take a prime ideal $\mathfrak{p}$ of $R$ lying over a rational prime $p$. 
	\begin{notation} For the rest of the paper, we fix two distinct odd primes $p$, and $l$. We will denote by $q$ the order of the residue field of $\mathbf{Z}[\cos \frac{2 \pi}{l}]$ at prime above $p$. Explicitly, $q=p^{\frac{\mathrm{ord} (p \mod l)}{2}}$, if $\frac{l-1}{\mathrm{ord (p \mod l)}}$ is odd, and $q=p^{\mathrm{ord} (p \mod l)}$, if $\frac{l-1}{\mathrm{ord (p \mod l)}}$ is even. 
	\end{notation}
	In the first case, $\mathfrak{p}$ remains inert in $\mathbf{Z}[\zeta_l]$, and $\mathbf{Z}[\zeta_l]/\mathfrak{p} = \mathbf{F}_{q^2}$, with complex conjugation inducing the involution $x \mapsto x^q$. If the reduction of the original hermitian form is nondegenerate, the reduction of the unitary group is the unitary group $\U(n-1,q)$ over the finite field $\mathbf{F}_q$ (Note that we follow the notation for unitary groups where $\mathbf{F}_q$ denotes the analogue of $\mathbf{R}$, not $\mathbf{C}$). \par
	In the second case, $\mathfrak{p}$ splits into $\mathbf{Z}[\zeta_l]$ into two primes $\mathfrak{q}$ and $\bar{\mathfrak{q}}$. Then we have $\mathbf{Z}[\zeta_l]/\mathfrak{p} = \mathbf{F}_q \oplus \mathbf{F}_q$, with the complex conjugation inducing an involultion which swaps the two factors. In this case, the reduction of the original finite free module becomes a pair of $\mathbf{F}_q$-vector spaces, and the reduction of the original hermitian form induces a pairing between these two vector spaces. Therefore, the reduction of the unitary group becomes $\GL(n-1,q)$.
	\begin{notation}
		Let $\mathbf{E}_q$ denote $\mathbf{Z}[\zeta_l]/\mathfrak{p}$ in either of the two cases: $\mathbf{E}_q=\mathbf{F}_{q^2}$ if $\frac{l-1}{\mathrm{ord (p \mod l)}}$ is odd, and $\mathbf{E}_q = \mathbf{F}_q \oplus \mathbf{F}_q$ if $\frac{l-1}{\mathrm{ord (p \mod l)}}$ is even. We will sometimes refer to these as 
		the unitary case', and 'the linear case' of the problem. No matter the case we are in, we can talk about the reduction of the integral unitary group being 'the unitary group over $\mathbf{E}_q$'.
	\end{notation}
	Let $X$ be a branched cyclic covering of $\mathbf{P}^1$ of some odd prime degree $l$. Then the etale cohomology $H^1_{\et}(X;\mathbf{F}_p)$ acquires a structure of an $\mathbf{F}_p[T]/(T^l - 1)$-module, with a fixed choice of a nontrivial deck transformation $\varphi : X \to X$ acting as $T$. This action induces a decomposition $H^1_{\et}(X; \mathbf{F}_p) = \ker (T-1) \oplus \ker \Big( \frac{T^l-1}{T-1} \Big)$, and moreover, $\ker(T-1)$ can be identified with $H^1_{\et}(X/\langle \varphi \rangle;\mathbb{Z}_p)$. However, $X / \langle \varphi \rangle \cong \mathbf{P}^1$, so this first summand vanishes, and hence $H^1_{\et}(X; \mathbf{F}_p)$ becomes an $\mathbf{F}_p[T] / (T^{l-1} + \ldots + T + 1) \cong \mathbf{F}_q^{\oplus k}$-module, where $q = p^{\mathrm{ord} (p \mod l)}$, $k = \frac{l-1}{\mathrm{ord} (p \mod l)}$. We will consider the components $H^1_{\et}(X; \mathbf{F}_p)^{\pi_j = 0}$, where the kernel of the corresponindg map $\pi_j : \mathbf{F}_p[T] / (T^{l-1} + \ldots + T + 1) \to \mathbf{F}_q$ acts trivially. \par  
	\begin{theorem} \label{main}
		If $n \geq l + 1$, then the image of the monodromy representation of $P_{n+1}$ on $H^1(X_{\bar{a}}; \mathbf{F}_p)^{\pi_j = 0}$ is either the group $\SlU(n-1,q) := \{A \in \GL(n-1,q^2) : A \cdot \mathrm{Fr}_q(A)^t = 1, \det A \in \mu_l \}$ or $\SlL(n-1,q) := \{A \in \GL(n-1,q) : \det A \in \mu_l\}$, depending on whether or not $\frac{l-1}{\mathrm{ord (p \mod l)}}$ is odd or even.
	\end{theorem}
	Our approach is inspired by an analogous theorem from \cite{cyclic}. There, an explicit model for the monodromy representation is provided, which they call a \emph{reduced Gassner representation}. We recall this model and its basic properties in \ref{gassnerbasic}. The main idea for the big monodromy proof is to find enough transvections in the image of monodromy. To do so, we need to pick a proper subsequence $(k_0, \ldots, k_i)$ of the monodromy vector with product $1$. In the explicit model, this subsequence will correspond to a coordinate subspace with the degenerate hermitian form. The isometry group of the degenerate hermitian form is non-reductive, and in the section \ref{degen} we prove that its unipotent radical is contained in the monodromy image. Then, in the section \ref{next_degen} we show how to lift the elements obtained on the previous step to automorphisms of the subspace of dimension one bigger and use a theorem of Zalesskii on groups generated by transvections (\cite{zalesskii}) to prove the base case of our big monodromy theorem. The main result then follows by a straighforward inductive argument. \par 
	Our ultimate motivation for this project lies in its applications in inverse Galois theory. The connection between big monodromy theorems and big Galois image theorems is well-known, and the latter is also a well-studied topic (see \cite{abvar}). In order to construct Galois extensions of $\mathbf{Q}$ with Galois groups $\PSL(n,q)$ and $\PSU(n,q)$, we will look at the Galois action on the torsion of the Jacobian of a cyclic cover of the projective line. In order for the Galois image to lie in $\PGL(n,q)$ we need the deck transformation of the curve to be defined over $\mathbf{Q}$, which will only be true under some congruence conditions on $n$. The main difficulty is that because the Weil pairing takes values in $\mu_l$, the arithmetic monodromy image might differ from the geometric monodromy. All of this is a topic of our forthcoming work.
	\section{Gassner representations} \label{gassnerbasic}
	In geometric group theory, the \emph{Gassner representation} is naturally defined as a homomorphism $g^{\mathrm{univ}}_n : P_{n+1} \to \GL_{n+1}(\mathbf{Z}[T_0^{\pm 1}, T_1^{\pm 1}\ldots, T_n^{\pm 1}])$\footnote{Do I need to provide explicit formulas here??}. It is shown in \cite[section 3.5]{cyclic} that the Gasssner representation decomposes as a direct sum of a one-dimensional trivial representation and a \emph{reduced Gassner representation} $\bar{g}_n^{\mathrm{univ}} : P_{n+1} \to \GL_n(\mathbf{Z}[T_0^{\pm 1}, T_1^{\pm 1}\ldots, T_n^{\pm 1}])$. Given a ring $R$ and invertible elements $t_0, \ldots, t_n \in R^{\times}$, let us define a \emph{specialized reduced Gassner representation} as a representation $\bar{g}_n(t_0, \ldots, t_n) : P_{n+1} \to \GL_{n-1}(R)$ obtained from the Gassner representation by a base change along the map $\mathbf{Z}[T_0^{\pm 1}, T_1^{\pm 1}\ldots, T_n^{\pm 1}] \to F$, $T_i \mapsto t_i$. It turns out that we can obtain the monodromy representation on the cohomology of cyclic branched coverings as a specialized reduced Gassner representation: 
	\begin{theorem} 
		Let $\mathfrak{p}$ vary over the set of primes of $\mathbf{Z}[\cos(\frac{2\pi}{l})]$ above $p$. Then the monodromy representation on the cohomology of cyclic branched coverings $X_{\bar{a}, \bar{k}}$ is isomorphic, as an $\mathbf{F}_p[T]/(T^{l-1} + \ldots + 1) \equiv \bigoplus_{\mathfrak{p}} \mathbf{E}_q$-module and a $P_{n+1}$-representation, to $\bigoplus_{\zeta }\bar{g}_n(\zeta^{k_0}, \zeta^{k_1}, \ldots, \zeta^{k_n})$, where $\zeta$ denotes, in the $\mathfrak{p}$-component, the reduction of a fixed primitive $l$-th root of unity mod $\mathfrak{p}$.
	\end{theorem}
	This is essentially proposition 24 of \cite{cyclic}. We note that even though the result is stated there for $\mathbb{Q}$-coefficients, the proof only uses that $l$ in invertible in the ring of coefficients. \par 
	We denote by $s_i$ the standard generators of the braid group. Then $s_i^2$ is a pure braid. The following result is Lemma 14 from \cite{cyclic}:
	\begin{lemma} \label{explicit}There is a basis  $\varepsilon_0, \cdots,  \varepsilon_{n-1}$ in the reduced Gassner representation $\bar{g}_n^{\mathrm{univ}}$, such that the action of  $s_i^2$ has the matrix form
		\[\begin{pmatrix} 1 & 0 & 0 \\ T_i (1-T_{i+1}) & T_iT_{i+1} & 1-T_i \\ 0
			& 0 & 1 \end{pmatrix} \oplus 1_{n-3} , \] where the $3\times 3$ matrix
		is with respect  to the basis elements $\varepsilon_{i-1},  \varepsilon_i, \varepsilon_{i+1}$ and
		$s_i^2$ acts as identity on the basis elements $\varepsilon_j$ for the
		other indices $j$. In particular, $s_i^2$ are \emph{generalized reflections} in the sense that $s_i^2-1$ has rank $1$.
	\end{lemma}
	In lemma 15, they estabilish an existence of a skew-hermitian form $h$ with respect to the involution $T_i^* = T_i^{-1}$ of the ring $\mathbf{Z}[T_0^{\pm 1}, T_1^{\pm 1}\ldots, T_n^{\pm 1}]$ on the reduced Gassner representation such that the determinant of the Gram matrix $h$ with respect to the $\varepsilon$-basis is $\frac{1 - T_0 T_1 \ldots T_n}{(1-T_0)(1-T_1) \ldots (1-T_n)}$. For our purposes, we note that the determinant of the matrix from lemma 14 specialised to $\mathbf{F}_q$ is a root of unity $\zeta^{k_i + k_{i+1}}$ and that the skew-hermitian form $h$ continues to be nondegenerate in the specialisation if and only if $k_0 + k_1 + \ldots + k_n \neq 0$. On the level of coverings $X_{\bar{a}, \bar{k}}$, these are the curves $y^l = \prod_{i=0}^n (x-a_i)^{k_i}$ whose smooth proper model becomes ramified over $\infty$. \par 
	Note that in the linear case, we can consider the Gassner representation into $\GL_{n-1}(\mathbf{F}_q \oplus \mathbf{F}_q)$ as a pair of representations over $\mathbf{F}_q$. The nondegeneracy of the hermitian form in this case means that the pairing in induces between the two representations is nondegenerate, so that they are dual to each other. 
	\begin{proposition} \label{irreducibility}
		The specialized reduced Gassner representation is absolutlely irreducible if $t_0 \ldots t_n \neq 1$. Otherwise, it has a one-dimensional space of invariant vectors, the quotient by which is absolutely irreducible.
	\end{proposition}
	\begin{proof}
		Let $V$ be a subrepresentation of the specialized reduced Gassner representation, which contains a non-invariant vector $v$. By inspection, for any $i$, the image of $s_i^2-1$ is spanned by $\varepsilon_i$, and the only basis vectors on which $s_i^2-1$ is nonzero, are $\varepsilon_{i \pm 1}$. Because $v$ is non-invariant, we may find an index $j$ such that $(s_j^2 - 1) v \neq 0$, and hence we get that $\varepsilon_j \in V$. Applying $s_{j-1}^2 - 1$, we get $\varepsilon_{j-1} \in V$, applying  $s_{j+1}^2 - 1$, we get $\varepsilon_{j+1} \in V$, and contuning in this way, we get that $V$ is the whole specialized reduced Gassner representation. \par 
		Now let us find invariant vectors. Let $u = x_0\varepsilon_0+ \ldots x_{n-1} \varepsilon_{n-1}$ be a vector such that $s_i^2 u = u$. This unfolds to $t_i(1-t_{i+1})x_{i-1} + t_it_{i+1}x_i-t_ix_{i+1}=0$. Starting from the first equation $t_0t_1x_0-t_0x_1=0$, we see that every coordinate can be expressed in terms of previous ones, so that the space of invariants is at most one-dimensional. In fact, solving all equations except the last one, we obtain $x_i = 1 - t_0 \ldots t_{i-1}$, and the last equation $t_{n-1}(1-t_n)x_{n-1}+t_{n-1}t_nx_n=0$ becomes equivalent to $1-t_0\ldots t_{n-1}=0$.
	\end{proof}
	Note also that in \cite[section 5.1]{cyclic}, the vector $(1-t_0)\varepsilon_0 + \ldots + (1-t_0 \ldots t_{n-1}) \varepsilon_{n-1}$ is found to exactly span the kernel of the hermitian form in the degenerate case. \par 
	We will require a slightly stronger property of the Gassner representation than mere irreducibility. We will need to know that is does not admit even a proper nonzero invariant $\mathbf{F}_p$-vector subspace. By the following general lemma, it amounts to the fact that the representation in question is not writable over a proper subfield of $\mathbf{F}_q$, which is true, because the determinants of $s_i^2$ generate the whole multiplicative group $\mathbf{F}_q^{\times}$.
	\begin{lemma} \label{invaddsubgp}
		Let $G$ be a group, and let $\rho$ be an absolutely irreducible representation of $G$ over a finite field $\mathbf{F}$ of characteristic $p$, which is not writable over a proper subfield of $\mathbf{F}$. Then $\rho$ admits no proper nonzero invariant $\mathbf{F}_p$-vector spaces.
	\end{lemma}
	\begin{proof}
		Let $\rho'$ be an $\mathbf{F}_p$-representation of $G$ obtained by the restriction of scalars from $\rho$. Then $\rho' \otimes \mathbf{F} = \bigoplus_{\sigma \in \Gal(\mathbf{F}/\mathbf{F}_p)} \rho^{\sigma}$. I claim that $\rho^{\sigma}$ are a family of distinct irreducible representations of $G$. Indeed, assume that there exists $\sigma_0 \neq 1$ such that $\rho^{\sigma_0} \cong \rho$. Then the (Frobenius) character of $\rho$ belongs to the field $\mathbf{F}^{\langle \sigma_0 \rangle}$. Therefore, by \cite[theorem 9.14]{isaacs}, $\rho$ descends to $\mathbf{F}^{\langle \sigma_0 \rangle}$, contradicting the assumption. \par
		The claim now follows, because the only subrepresentations of $\rho' \otimes \mathbf{F}$ are the direct sums of $\rho^{\sigma}$'s, but clearly no such proper nonzero sum descends to $\mathbf{F}_p$.
	\end{proof}
	\section{Gassner representations with degenerate Hermitian forms} \label{degen}
	The following is contained in the proposition 21 of \cite{cyclic}:
	\begin{proposition}\label{nontriviality1}
		Let $\bar{g}_n(t_0, \ldots, t_n)$ be a specialized reduced Gassner representation over $\mathbf{E}_q$ such that $t_0 \ldots t_n = 1$ and $1 - t_0 \in \mathbf{E}_q^{\times}$. Let $\Delta' = (s_1 \ldots s_n)(s_1 \ldots s_{n-1}) \ldots (s_1s_2) s_1$. Then the image $[s_0^2, \Delta'^2]$ is a nontrivial element of the unipotent radical of the degenerate unitary group. 
	\end{proposition}
	Note that the Gassner representations are embedded in one another: if $\varepsilon_0, \ldots, \varepsilon_{n-1}$ is the standard basis of $\bar{g}_n(t_0, \ldots, t_n)$, then $\varepsilon_1, \ldots, \varepsilon_{n-1}$ is the standard basis for $\bar{g}_{n-1}(t_1, \ldots, t_{n-1})$, meaning that the action of the subgroup $P_n \subset P_{n+1}$ of braids where the zeroth strand is not involved preserves the hyperplane $\langle \varepsilon_1, \ldots, \varepsilon_n \rangle$ and its action there is by the previous Gassner representation. \par 
	\begin{definition}
		Let $V$ be a vector space over a field $F$, let $v \in V$, let $W$ be a hyperplane of $V$ not containing $v$, and let $\xi$ be a linear functional on $W$. Writing a general vector in $V$ as $cv + w$, $c \in F, w \in W$, we can consider a linear transformation $cv + w \mapsto (c + \xi(w))v +w$. Linear transformations of such form for varying $\xi$ will be called \emph{transvections along $v$}, and denoted $T_v(\xi)$. This definition is independent of $W$ in an obvious sense.
	\end{definition}
	We can now state what the unipotent radicals of the relevant unitary groups are. For the unitary case, if $V$ is vector space over a field which carries an Hermitian form with one-dimensional kernel spanned by $v$, then the unipotent radical of the isometry group of this form consists of all the transvections along $v$. For the linear case, if $V$ and $W$ are two vector spaces over the same field, which carry a pairing with one-dimensional kernels on each side, spanned by $v$ and $w$, respectively, then the unipotent radical of the isometry group of this pairing consists of pair of transvections along $v$ and $w$. \par 
	Specializing the discussion to the case $V=\bar{g}_n(t_0, \ldots, t_n)$ under the assumptions of \ref{nontriviality1}, we may pick $W = \langle \varepsilon_1, \ldots, \varepsilon_n \rangle$, and identify the transvections along $v$ with linear functionals on $W$. I claim that the conjugation action of the isometries on transvections is then identified with the smaller Gassner representation (or technically, its dual, but because it carries a nondegenerate hermitian form, the distinction will be ignored). Indeed, for any $\sigma \in P_n$, $\sigma^{-1} T_v(\xi) \sigma (w) = \sigma^{-1}(\sigma(w)+\xi(\sigma w)v)=w+\sigma^*(\xi)(w)\sigma^{-1}(v)=w+\sigma^*(\xi)(w)v$, because $v$ is invariant. \par 
	Now let us consider the intersection of the image of $\bar{g}_n(t_0, \ldots, t_n)$ with the unipotent radical of the unitary group. Let us split the discussion into two cases. In the unitary case, this intersection is nonzero by \ref{nontriviality1}, and because it is a subgroup, we may consider it as an additive subgroup of $\bar{g}_{n-1}(t_1, \ldots, t_{n-1})$, invariant under the action of $P_n$. But by lemma \ref{invaddsubgp}, there are no proper nonzero invariant additive subgroups, so we conclude that the image of $\bar{g}_n(t_0, \ldots, t_n)$ contains the unipotent radical in this case. \par 
	For the linear case, where the representation is over $\mathbf{E}_q = \mathbf{F}_q \oplus \mathbf{F}_q$, note that the formula given in \cite[Proposition 21]{cyclic}, gives that $\xi(\varepsilon_2) = t_1^{-1}t_2^{-1}(1-t_1)$, so this transvection in nontrivial over both copies of $\mathbf{F}_q$. Now, when we consider the intersection of Gassner image with the unipotent radical as an additive group with an action of $P_n$, it becomes $\bar{g}_{n-1}(\pi_1(t_1), \ldots, \pi_1(t_n)) \oplus \bar{g}_{n-1}(\pi_2(t_1), \ldots, \pi_2(t_n))$, where $\pi_i : \mathbf{F}_q \oplus \mathbf{F}_q \to \mathbf{F}_q$ are projections. Therefore in this case, the Gassner image contains the unipotent radical if $\bar{g}_{n-1}(\pi_1(t_1), \ldots, \pi_1(t_n)) \not \cong \bar{g}_{n-1}(\pi_2(t_1), \ldots, \pi_2(t_n))$, but may be the graph of an isomorphism between $\bar{g}_{n-1}(\pi_1(t_1), \ldots, \pi_1(t_n))$ and $\bar{g}_{n-1}(\pi_2(t_1), \ldots, \pi_2(t_n))$ inside their direct sum, if such exists. 
	\section{Gassner representations with degenerate coordinate hyperplanes} \label{next_degen}
	In the next step, we have a significant departure from the argument in \cite{cyclic}. In the original article, the proof of proposition 23 relies on the Zariski density of the monodromy image in the unitary group, which cannot be used over a finite field. We argue in a different way. The main new insight is that we can obtain transvections in the image of the "next" Gassner representation $\bar{g}_{n+1}(t_0, \ldots, t_{n+1})$ as commutators of liftings of transvections from the degenerate Gassner representation $\bar{g}_n(t_0, \ldots, t_n)$. 
	\begin{proposition} \label{next_transvect}
		Assume that $q$ is odd and $n > 2$. Let $\bar{g}_{n+1}(t_0, \ldots, t_{n+1})$ be a specialized reduced Gassner representation over $\mathbf{E}_q$ such that $t_0 \ldots t_n = 1$, $t_0 -1 \in \mathbf{E}_q^{\times}$, and $t_{n+1} - 1 \in \mathbf{E}_q^{\times}$. Then there exists an isotropic vector $v$\footnote{This means not $(0,u)$, $(u,0)$} such that the image of $\bar{g}_{n+1}(t_0, \ldots, t_{n+1})$ contains all unitary transvections along $v$.
	\end{proposition}
	\begin{proof} 
		By assumption, the subspace spanned by $\varepsilon_0, \ldots, \varepsilon_{n-1}$, is the space of the degenerate Gassner representation $\bar{g}_n(t_0, \ldots, t_n)$, hence the hermitian form $h$ has a one-dimensional kernel on this hyperplane. We take $v$ to be the spanning vector of the kernel. 
		 Let us take an element $T_v(\phi)$ in the unipotent radical of $U(h|_{\langle \varepsilon_0, \ldots, \varepsilon_{n-1} \rangle })$ , and explicitly calculate what are the possible liftings $\widetilde{T}$ of $T_v(\phi)$ to $U(h)$. \par 
		Pick a vector $\hat{\phi}$ such that $\phi(u) = h(u,\hat{\phi})$ for $u \in \langle \varepsilon_0, \ldots, \varepsilon_{n-1} \rangle$. The image $\widetilde{T}(\varepsilon_n)$ has to satisfy $h(u,\widetilde{T}(\varepsilon_n))=h(\widetilde{T}^{-1}(u),\varepsilon_n) = h(u - \phi(u) v, \varepsilon_n)=h(u, \varepsilon_n)-\phi(u)h(v,\varepsilon_n)=h(u,\varepsilon_n)-h(u, \hat{\phi})h(v,\varepsilon_n)=h(u,\varepsilon_n)-h(u, \overline{h(v,\varepsilon_n)}\hat{\phi})$, which implies that $\widetilde{T}(\varepsilon_n)-\varepsilon_n+\overline{h(v,\varepsilon_n)}\hat{\phi}$ can be any multiple of $v$. For convenience, replace $v$ with $\overline{h(v,\varepsilon_n)}^{-1} v$. \par 
		Let us pick a convenient basis of $\bar{g}_n(t_0, \ldots, t_{n+1})$. We will take the first basis element to be $v$, we will let $e_2, \ldots, e_{n}$ to be an orthonormal basis of $\langle \varepsilon_1, \ldots, \varepsilon_m \rangle$, and we will let the last basis element to be $\varepsilon_n$. Then we can write $\hat{\phi} = \alpha_2 e_2 + \ldots \alpha_{n-1} e_{n-1}$, so that $\widetilde{T}$ has the form 
		\begin{equation}\begin{pmatrix}
			1 & \overline{\alpha_2} & \overline{\alpha_3} & \ldots & \overline{\alpha_n} & \lambda \\
			0 & 1 & 0 & \ldots & & \alpha_2 \\
			0 & 0 & 1 & \ldots & & \alpha_3 \\
			\vdots & & & & & \vdots \\
			0 & 0 & 0 & \ldots & & \alpha_n \\
			0 & 0 & 0 & \ldots & & 1
		\end{pmatrix} \label{extension} \end{equation}
		To find unitary transvections with direction $v$ in the image of $\bar{g}_n(t_0, \ldots, t_{n+1})$, consider two such elements $\widetilde{T}_1$, $\widetilde{T}_2$ with distinct first rows $1, \overline{\alpha_2}, \ldots$ and $1, \overline{\beta_2}, \ldots$. For simplicity, assume that $\sum_i \alpha_i \overline{\alpha_i} = \sum_i \beta_i \overline{\beta_i}=1$. Then a direct computation shows that 
		$$\begin{pmatrix}
			1 & \overline{\alpha_2} & \overline{\alpha_3} & \ldots & \overline{\alpha_n} & \lambda \\
			0 & 1 & 0 & \ldots & & \alpha_2 \\
			0 & 0 & 1 & \ldots & & \alpha_3 \\
			\vdots & & & & & \vdots \\
			0 & 0 & 0 & \ldots & & \alpha_n \\
			0 & 0 & 0 & \ldots & & 1
		\end{pmatrix}^{-1} =  \begin{pmatrix}
			1 & -\overline{\alpha_2} & -\overline{\alpha_3} & \ldots & -\overline{\alpha_n} & 1-\lambda \\
			0 & 1 & 0 & \ldots & & -\alpha_2 \\
			0 & 0 & 1 & \ldots & & -\alpha_3 \\
			\vdots & & & & & \vdots \\
			0 & 0 & 0 & \ldots & & -\alpha_n \\
			0 & 0 & 0 & \ldots & & 1
		\end{pmatrix}$$
		and 
		$$\Big[ \begin{pmatrix}
			1 & \overline{\alpha_2} & \overline{\alpha_3} & \ldots & \overline{\alpha_n} & \lambda \\
			0 & 1 & 0 & \ldots & & \alpha_2 \\
			0 & 0 & 1 & \ldots & & \alpha_3 \\
			\vdots & & & & & \vdots \\
			0 & 0 & 0 & \ldots & & \alpha_n \\
			0 & 0 & 0 & \ldots & & 1
		\end{pmatrix} 
		\begin{pmatrix}
			1 & \overline{\beta_2} & \overline{\beta_3} & \ldots & \overline{\beta_n} & \mu \\
			0 & 1 & 0 & \ldots & & \beta_2 \\
			0 & 0 & 1 & \ldots & & \beta_3 \\
			\vdots & & & & & \vdots \\
			0 & 0 & 0 & \ldots & & \beta_n \\
			0 & 0 & 0 & \ldots & & 1
		\end{pmatrix} \Big] = \begin{pmatrix}
			1 & 0 & 0 & \ldots & 0 & \sum_i \overline{\alpha_i}\beta_i - \alpha_i \overline{\beta_i} \\
			0 & 1 & 0 & \ldots & & 0 \\
			0 & 0 & 1 & \ldots & & 0 \\
			\vdots & & & & & \vdots \\
			0 & 0 & 0 & \ldots & & 0 \\
			0 & 0 & 0 & \ldots & & 1
		\end{pmatrix}$$
		I claim that for $n > 2$, we may pick $\widetilde{T}_1$ and $\widetilde{T}_2$ in such a way that $ \sum_i \overline{\alpha_i}\beta_i - \alpha_i \overline{\beta_i} $ is any imaginary scalar. Indeed, let $\xi$ be an imaginary scalar. Because the norm map is surjective, there is an $\eta$ such that $\frac{\xi \overline{\xi}}{4} + \eta \overline{\eta}=1$. Then we take $(\alpha_2, \alpha_3, \ldots, \alpha_n) = (\frac{\xi}{2}, \eta, 0,\ldots, 0)$, $(\beta_2, \beta_3, \ldots, \beta_n) = (1,0, \ldots, 0)$.
	\end{proof}
	\begin{definition}
		Let $G$ be any group, $R$ be any commutative ring, and $\rho : G \to \GL_d(R)$ be a representation of $G$ on a free $R$-module. We say that $\rho$ is primitive, if there does not exist a decomposition of $R^d$ as a direct sum of free $R$-modules setwise preserved by $G$. Equivalently, $\rho$ is primitive if it is not induced from a finite index subgroup of $G$. 
	\end{definition}
	\begin{lemma} \label{next_prim}
		Under the assumptions of the proposition \ref{next_transvect}, the Gassner representation $\bar{g}_n(t_0, \ldots, t_{n+1})$ is primitive.
	\end{lemma}
	\begin{proof}
		Let $V_1 \oplus \ldots \oplus V_k$ be a direct sum decomposition preserved by $P_{n+2}$. First, I claim that any transvection in the image of the Gassner representation acts on the set $\{V_1, \ldots, V_k\}$ trivially. Without loss of generality, assume that $\tilde{T}$ is a transvection with direction $u$ in the image of the Gassner representation such that $\widetilde{T} V_1 \subseteq V_2$. If for any $v_1 \in V_1$, $v_1 + \xi(v_1)u \in V_2$, then in particular, $u \in V_1 + V_2$. This means that $V_1 + V_2$ is a $\widetilde{T}$-invariant subspace. Because $\widetilde{T}$ must act on  $\{V_1, \ldots, V_k\}$  by a permutation, we conclude that $\widetilde{T}(V_2)=V_1$. But this implies that $\widetilde{T}$ is an element of even order, while in reality its order is $p$ -- the characteristic of $\mathbf{F}_q$, which is assumed to be odd.  \par 
		Now pick one of the spaces in this decomposition, say $V_i$, does not contain the vector $v$ from proposition \ref{next_transvect}. We can pick two unipotent transformations $\widetilde{T}_1, \widetilde{T}_2$, which, when written in the form (\ref{extension}), have the same $\alpha_2, \ldots, \alpha_n$, but different $\lambda$'s, with the first $\lambda$ nonzero. Pick a nonzero vector $w$ in $V_i$. Then if $w$ has nonzero coordinate in the last vector $\varepsilon_n$, we have $\widetilde{T}_1(w)-\widetilde{T}_2(w) \in V_i$, but at the same time $\widetilde{T}_1(w)-\widetilde{T}_2(w) \in V_i$ is a nonzero multiple of $v$, which is impossible. On the other hand, if $w$ has zero coordinate in $\varepsilon_n$, then $\widetilde{T}_1(w)-w$ is both in $V_i$ and a nonzero multiple of $v$, which is again impossible. \par 
	\end{proof}
	We obtain the desired surjectivity from a theorem of Zalesskii:
	\begin{theorem} \cite{zalesskii}
		Suppose $G \subset \GL(n,q)$ is an irreducible group generated by transvections, and suppose that $n>2$ and $q$ is odd. Then $G$ is conjugate in $\GL(n, k)$ to one of the groups $\SL(n,q')$, $\Sp(n,q')$ or $\SU(n,q')$, where $q'$ divides $q$.
	\end{theorem}
	\begin{theorem} \label{next_image}
		Under the assumptions of the proposition \ref{next_transvect}, the image of $\bar{g}_n(t_0, \ldots, t_{n+1})$ contains $\SU(h)$.
	\end{theorem}
	\begin{proof}
			Let us say that an isotropic vector $u$ is \emph{good}, if all the unitary transvections with direction $u$ are in the image of $\bar{g}_n(t_0, \ldots, t_{n+1})$ . Clearly, the action of $P_{n+2}$ preserves the set of good isotropic vectors. By proposition \ref{next_transvect}, there exist a good isotropic vector. Because $\bar{g}_n(t_0, \ldots, t_{n+1})$ is irreducible, the $P_{n+2}$-orbit of a good isotropic vector spans the whole space. Let $H$ be a group generated by all the unitary transvections in the image of $\bar{g}_n(t_0, \ldots, t_{n+1})$. It suffices to prove that $H$ is irreducible, for then $H=\SU(h)$ by Zalesskii's theorem. \par 
			Let $W$ be a proper nonzero $H$-invariant subspace. Then for any good isotropic vector $w$, $W$ is invariant with respect to any unitary transvection with direction $w$. I claim that $w$ is either contained in $W$ or is orthogonal to it. Indeed, assume that $w \not \in W$. Then for any unitary transvection $T_w(\xi)$ with direction $w$, and for any $u \in U$, $h(u,w) w = T_w(\xi)(u)-u \in W$, hence $h(u,w)=0$. So, if $W$ is any $H$-invariant subspace, then the set of good isotropic vectors is contained in $W \cup W^{\perp}$.  \par 
			Let $U_1$ be a minimal proper nonzero $H$-invariant subspace. Then $U_1 \cap U_1^{\perp}$ is also $H$-invariant, so by minimality either $U_1 \cap U_1^{\perp} = 0$, and $U_1$ is nondegenerate, or $U_1 \cap U_1^{\perp} = U_1$, and $U_1$ is isotropic. I claim that the second case is actually impossible. Indeed, by the previous paragraph, the set of good isotropic vectors will be contained in $U_1 \cup U_1^{\perp} = U_1^{\perp}$. This contradicts the fact that isotropic vectors span the whole space.
			\par 
			Given that $U_1$ is nondegenerate, pick $U_2 \subset U_1^{\perp}$ minimal $H$-invariant, and continue in this fashion. We will get an orthogonal decomposition $U_1 \oplus \ldots \oplus U_k$, where all the subspaces $U_i$ are $H$-irreducible. Note that the set of good isotropic vectors is contained in $U_i$. I claim that $G$ preserves this decomposition setwise. Indeed, pick an $i = 1, \ldots, k$, and pick a good isotropic vector $u \in U_i$. For any $g \in G$, $gu$ is a good isotropic vector, hence there is a $j$ such that $gu \in U_j$. Because $gU_i \cap U_j$ is nonzero and $H$-invariant, we must have $gU_i = U_j$. We conclude from lemma \ref{next_prim} that $k=1$ and $H$ was acting irreducibly to begin with.
	\end{proof}
	\section{Finishing the proof}
	Before finishing the proof, we deduce another elementary lemma from Zalesskii's theorem:
	\begin{lemma} \label{hyperplane}
		Let $F$ be a field of odd characteristic, let $n > 2$, let $E$ be a quadratic extension of $F$, let $V$ be an $n$-dimensional nondegenerate hermitian space over $E$, and let $V_0 \subseteq V$ be a nondegenerate hyperplane. Then any irreducible subgroup of $\SU(V)$ containing $\SU(V_0)$, is $\SU(V)$ itself.
	\end{lemma}
	\begin{proof}
		Indeed, let $G$ be such a subgroup, containing the unitary group of a hyperplane $v^{\perp}$ for some $v$. Let us say that a non-isotropic vector $u$ is \emph{good}, if $G$ contains the unitary group of $u^{\perp}$. Clearly, $G$ acts on the set of good vectors, so they have to span the whole space, because $G$ is irreducible. Let $H$ be a group generated by all the $\SU(u^{\perp})$ for $u$ good. By Zalesskii's theorem, it suffices to prove that $H$ is irreducible. Let $U$ be an $H$-invariant subspace. For any good vector $u$, $U$ has to be $\SU(u^{\perp})$-invariant, which means that $U$ is either $\langle u \rangle $ or $u^{\perp}$. But of course, for any other good vector $u'$, $U$ would have to be either either $\langle u' \rangle $ or $u'^{\perp}$. This is only possible if $n=2$ and $u,u'$ form an orthogonal basis.
	\end{proof} 
	\begin{theorem} 
		\begin{enumerate}
			\item Let $\bar{g}_n(t_0, \ldots, t_n)$ be a specialized reduced Gassner representation over $\mathbf{E}_q$ such that $t_0 \ldots t_n \neq 1$, and there exist a proper subsequence of $(t_0, \ldots, t_n)$ of length $\geqslant 3$ with product $1$, which does not contain $1$'s. Then the image of $\bar{g}_n(t_0, \ldots, t_n)$ contains $\SU(h)$;
			\item Let $\bar{g}_n(t_0, \ldots, t_n)$ be a specialized reduced Gassner representation over $\mathbf{E}_q$ such that all $t_i \in \mu_d$, $d \geq 3$, $t_i \neq 1$ and $n \geqslant d$. Then the image of $\bar{g}_n(t_0, \ldots, t_n)$ contains $\SU(h)$.
		\end{enumerate}
	\end{theorem}
	\begin{proof}
		For part 1), fix a reordering of the branch points, such that there exists an $m \leqslant n$ such that $t_0 t_1 \ldots t_m = 1$, and $t_0 \neq 1$. In this reordering, pick a maximal $m$ such that $t_0 t_1 \ldots t_m = 1$. By theorem \ref{next_image}, the image of the subgroup $P_{m+2}$ 
		contains the whole unitary group of the subspace $\langle \varepsilon_0, \ldots, \varepsilon_m \rangle$. We prove by induction that the image of the subgroup $P_{m+2+i}$ contains the whole unitary subgroup of the subspace $\langle \varepsilon_0, \ldots, \varepsilon_{m+i} \rangle$. Assume that we know the theorem for the braid subgroup generated by the first $m+1+i$ strands, $i \geqslant 1$ and we prove it for the braid subgroup generated by the first $m+2+i$ strands. Consider the pure braid groups $P_{m+1+i}$ given by omitting the last strand. It operates on the subset $\varepsilon_0, \ldots, \varepsilon_{m+i}$ of the standard basis for the Gassner representation. By the inductive assumption, the image of $P_{m+1+i}$ will be the whole unitary group of a hyperplane $\langle \varepsilon_0, \ldots, \varepsilon_{m+i-1} \rangle$. Because $m$ was chosen maximal with respect to the property that $t_0 t_1 \ldots t_m = 1$, we have that $t_0 t_1 \ldots t_{m+i+1} \neq 1$, and therefore, $\bar{g}_{m+i+1}(t_0, t_1, \ldots, t_{m+i+1})$ is irreducible. Hence, by lemma \ref{hyperplane}, the image of $\bar{g}_{m+i+1}(t_0, t_1, \ldots, t_{m+i+1})$ contains the whole unitary group of $\langle \varepsilon_0, \ldots, \varepsilon_{m+i+1} \rangle$. \par 
		For part 2), assume that $n \geq d$. Order $t_0, \ldots, t_n$ in such a way that elements equal to a fixed value form a block of consecutive entries. By the pigeonhole principle, among the $d+2$ elements $1, t_0, t_0t_1, \ldots, t_0 \ldots t_d \in \mu_d$, there will be two distinct pairs of indices $i < j$ and $k < l$ such that $t_0 \ldots t_i = t_0 \ldots t_j$ and $t_0 \ldots t_k = t_0 \ldots t_l$. Without loss of generality, assume $k \geqslant i$. Note that in fact $j \neq i+1$ and $l \neq k+1$ because all the entries are assumed to not be $1$. If $j \geq i + 3$, then part 1) applies to the subsequence $(t_{i+1}, \ldots t_j)$. If $l \geq k+3$, then part 1) applies to the subsequence $(t_{k+1}, \ldots t_l)$. \par 
		Assume $j=i+2$ and $l=k+2$, If $k \geq j$, then part 1) applies to the subsequence $(t_{i+1}, t_{i+2}, t_{k+1}, t_{k+2})$. So assume $k < j = i+2$. If $k=i$, then the pairs $(i,j)$ and $(k,l)$ are not distinct, contrary to the assumption. It remains to deal with the most difficult case $k=i+1$.\par 
		In this case we have $t_{i+1}t_{i+2}=1$, $t_{i+2}t_{i+3}=1$. This implies $t_{i+1}=t_{i+3}$, hence, by the assumption on the ordering, $t_{i+2}=t_{i+1}$ also, and therefore $t_{i+1}=-1$. Let us relabel the elements in such a way that $\{t_{i+1}, t_{i+2}, t_{i+3}\}$ become $\{t_{d-2}, t_{d-1}, t_d\}$. Consider the images of the $d-1$ elements $1, t_0, t_0t_1, \ldots, t_0t_1 \ldots t_{d-3}$ in the group $\mu_d / \{\pm 1 \}$. As $d \ge 4$, we have $d-1 > \frac{d}{2}$, so there exist some pair of indices $i' < j'$ such that $t_{i'+1} \ldots t_{j'} = \pm 1$. If $t_{i'+1} \ldots t_{j'} = 1$, then part 1) applies to $(t_{i'+1}, \ldots, t_{j'}, t_{d-1}, t_d)$. If $t_{i'+1} \ldots t_{j'} = -1$, then part 1) applies to $(t_{i'+1}, \ldots, t_{j'}, t_{d-2}, t_{d-1}, t_d)$.
	\end{proof}
	\begin{remark}
		Note that the assumptions in part 2) are optimal, because for the sequence $(\zeta, \zeta, \ldots, \zeta, \zeta^{-1})$, where $\zeta$ appears $d-1$ times, there is no subsequence of length $\geq 2$, for which part 1) applies.
	\end{remark}
	\begin{proof} (of theorem \ref{main})
		We have just proven that the image of the reduced Gassner representation contains either $\SU(n,q)$ or $\SL(n,q)$. To determine the image completely as a subgroup of $\U(n,q)$ or $\GL(n,q)$ we only need to know the values which can be taken by the determinant. By lemma \ref{explicit}, the determinant on the generators takes value $t_it_{i+1}$, which is an $l$-th root of unity. We can always reorder the elements $t_0, \ldots, t_n$ in such a way that, say, $t_0 t_1 \neq 1$, and therefore, the group multiplicatively generated by determinants will be exactly $\mu_l$.
	\end{proof}

\end{document}